\documentclass{article}
\usepackage{amssymb,amsmath, amsthm}
\usepackage{bm}
\newtheorem{theo}{Theorem}
\newtheorem{lem}[theo]{Lemma}
\newtheorem{rem}[theo]{Remark}

\newtheorem{coro}[theo]{Corollary}
\newtheorem{prop}[theo]{Proposition}

\title{Multivariate arithmetical functions and  vector calculus}
\author{Yusuke Fujisawa}

\begin{document}

\maketitle

\begin{center}
SSS Shigakoen \\
Tashiro building 402, 1-11-27, Ikeshita, \\ 
Chikusa-ku, Nagoya 451-0015, Japan \\
Mail: fujisawa.gifu@gmail.com 
\end{center}

\section{Introduction}

Let $d$ be a positive integer,  
$\mathbb{N}$ the set of positive integers, and 
$R$ a commutative ring with  unit.
We call a map $f:\mathbb{N}^d \rightarrow R$ an arithmetical function of $d$ variables. 
Arithmetical functions of several variables were studied by several authors. 
(See, e.g., \cite{alkan}, \cite{liskovets}, \cite{schwab}.) 
The motivation of this study is to consider vector calculus of arithmetical functions of $d$ variables. 
First, we shall show analogues of integral theorems after we define  analogues of exterior differential operators and boundary maps. 
Next,  an analogue of Poincar\'{e}'s lemma will be proved.  
Our theorem provides  necessary and sufficient conditions that some simultaneous difference equations have solutions.

Vector calculus is a strong tool in physics and 
discrete vector calculus  is developing in recent years. 
(See \cite{hefferlin},  \cite{robidoux}, \cite{schwalm},etc.) 
Also, discrete differential geometry has been established. 
The author studied under the influence of the theory of discrete vector calculus.
Our purpose is to construct vector calculus of arithmetical functions of $d$ variables. 
While  our investigation are essentially contained in \cite{robidoux}
 when $d=3$,  the results of this paper are partial generalization of results of discrete vector calculus.

Let $G$ be an abelian group and take distinct elements  $e_1, \cdots, e_d  \in G$. 
We assume that  subset $S \subset G$ has the following property;
\begin{align}\label{daiji}
a \in S \Longrightarrow a+e_i \in S.  
\end{align}
Moreover, let  $\mathcal{A}(S)$be the set of maps from $S$ to $R$, 
that is, 
\begin{align*}
\mathcal{A}(S) = \{ f: S \rightarrow R \}.
\end{align*}  
We fix  $G$, $ e_1, \cdots, e_d$,  and $S$.

We define $R$ moldules $\Omega^q(S)$ as follows. 
First, let $M$ be a free $\mathcal{A}(S)$ module of rank $d$ which is generated by formal elements $dx_1, \cdots, dx_d$ and 
consider the exterior algebra of $M$, that is, 
\begin{align*}
\bigwedge(M)=\bigoplus_{q=0}^{d} \bigwedge^q (M).
\end{align*}
 
Put $\Omega^q(S)= \bigwedge^q(M)$. Since 
$\{ dx_{i_1} \wedge dx_{i_2} \wedge \cdots \wedge dx_{i_q} |
1 \leq i_1< \cdots < i_q \leq d
 \}$ is a basis of $\Omega^q(S)$, 
 elements of $\Omega^q(S)$ are written in the form 
\begin{align*}
\sum_{1 \leq i_1 < \cdots < i_q \leq d}f_{i_1  \cdots  i_q }
\; dx_{i_1} \wedge \cdots \wedge dx_{i_q} 
\end{align*}
where $f_{i_1  \cdots  i_q } \in \mathcal{A}(S)$.
Thus, $\Omega^0(S)$ can be identified with $\mathcal{A}(S)$. 
Both $\Omega^q(S)$ and $\mathcal{A}(S)^{\binom{d}{q}}$ are $R$ modules, and there is 
a bijective correspondence between them.

Also, we construct   $R$ modules $N_q(S)$  which can be regarded as a set of geometric objects. 
Let $N_0(S)$ be the set of formal finite sums 
\begin{align*}
\sum r_{a} [a], \;\;\; (r_{a} \in R,  a \in S).
\end{align*}
 For $1 \leq q \leq d$, we define $N_q(S)$ as follows. 
Consider a formal element 
$A=[a:e_{i_1}, \cdots, e_{i_q}]$
where 
$a \in S$ and
$e_{i_1}, \cdots, e_{i_q} \in \{ e_1, \cdots, e_d\}, (i_1 < \cdots < i_q)$ are relatively distinct. 
$N_q(S)$ is the set of formal finite sums 
\begin{align*}
\sum r_{A} A, \;\;\; (r_A\in R, A=[a:e_{i_1}, \cdots, e_{i_q}]).
\end{align*}
If $d=3$ and $G=\mathbb{Z}^3$,  elements of $N_0(S)$, 
$N_1(S)$, $N_2(S)$, and 
$N_3(S)$ 
are formal sums of points, segments, parallelogram, and parallelepiped respectively.  

In Section 2, we define 
$R$ homomorphisms $D_q:\Omega^{q-1}(S) \rightarrow \Omega^{q}(S)$, 
$D'_{q}:N_{q}(S) \rightarrow N_{q-1}(S)$, $(1 \leq q \leq d)$, 
and 
$R$ bilinear maps $B_{q}:\Omega^q(S) \times N_q(S) \rightarrow R$
 $(0 \leq q \leq d)$. 
We sometimes write them   $D$, $D'$, and $B$ when no confusion can arise. 
In this paper,  $D$, $D'$, and  $B$ are said to be exterior differential operators, 
boundary maps, and integrals respectively. 
The following theorem is proved in Section 2.

\begin{theo}\label{stokes}
We have 
\begin{align*}
D_{q+1}D_q \equiv 0 \; (q=1, \cdots, d-1), \\
D'_{q}D'_{q+1} \equiv 0 \; (q=1, \cdots, d-1),
\end{align*}
and 
\begin{align*}
B_q(D_q \omega_{q-1}, \; A_{q})=B_{q-1}(\omega_{q-1}, \; D'_{q}A_{q})
, (q=1, \cdots, d)
\end{align*}
for $\omega_{q-1} \in \Omega^{q-1}(S)$ and $A_q \in N_q(S)$.
\end{theo}

The above theorem contains analogues of integral theorems of 
(discrete) vector calculus. 
When $d=3$ and $G=\mathbb{Z}^3$, essentially same results are written in \cite{robidoux}. 
$B_0$, $B_1$, $B_2$, and
$B_3$ are regarded as an values of scalar fields, 
line integrals of vector fields, surface integrals of vector fields, and
volume integrals of scalar fields respectively. 
Therefore, 
\begin{align*}
B_1(D_1 \omega_{0}, \; A_{1})=B_{0}(\omega_{0}, \; D'_{1}A_{1}), \\
B_2(D_2 \omega_{1}, \; A_{2})=B_{1}(\omega_{1}, \; D'_{2}A_{2}), 
\end{align*}
and 
\begin{align*}
B_3(D_3 \omega_{2}, \; A_{3})=B_{2}(\omega_{2}, \; D'_{3}A_{3}), 
\end{align*}
correspond to  potential theorem, Stokes' theorem, and  divergence theorem respectively.

In addition, we consider whether  
the sequence of  $R$ modules 
\begin{align*}
&\Omega^0(S) \xrightarrow{D_1} 
\Omega^1(S) \xrightarrow{D_2} 
\cdots \xrightarrow{D_q} \cdots
\xrightarrow{D_d} 
\Omega^d(S) \end{align*}
is exact or not. This question seems natural in mathematics.  
For $q=1, \cdots, d-1$, put 
\begin{align*}
H^q(S)= \frac{\textrm{Ker} \; (D_{q+1}:\Omega^{q}\rightarrow \Omega^{q+1})}
{\textrm{Im} \; (D_q:\Omega^{q-1}\rightarrow \Omega^{q})}.
\end{align*}

In Section 3, 
we assume that 
$G=\mathbb{Z}^d$, 
\begin{align*}
e_1=(1, 0, \cdots, 0), \;
e_2=(0, 1, 0, \cdots, 0), \;
\cdots, \;
e_d=(0, \cdots, 0, 1),
\end{align*}
and $S=\mathbb{N}^d$. 
Then, an element of $\mathcal{A}(S)=\mathcal{A}(\mathbb{N}^d)$ is an arithmetical function of $d$ variables. 
 We  determine $H^q(S)=H^q(\mathbb{N}^3)$ in this special case. 
If we put $H^0(\mathbb{N}^d)=\textrm{Ker}\; D_1$ and $H^d(\mathbb{N}^d)=\mathcal{A}(S)/\textrm{Im}D_d$
, then $H^0(\mathbb{N}^d)=R$ and $H^d(\mathbb{N}^d)=0$ (by the definition of $D$ in Section 2) . In Section 3, 
the following theorem 
which is an analogue of 
Poincar\'{e}'s lemma will be proved .

\begin{theo}\label{poincare}
\begin{align*}
H^q(\mathbb{N}^d)=
\begin{cases}
R \;\;\; \textrm{if $q=0$,} \\
0 \;\;\; \textrm{otherwise.}
\end{cases}
\end{align*}
\end{theo}

The discrete Poincar\'{e} lemma has been considered in \cite{desbrun}. 
However, the above is different from it.

For an arithmetical function $a(n_1, n_2, n_3)$ of three variables, 
put 
\begin{align*}
\partial_1 a(n_1, n_2, n_3)=a(n_1+1, n_2, n_3) - a(n_1, n_2, n_3) \\
\partial_2 a(n_1, n_2, n_3)=a(n_1, n_2+1, n_3) - a(n_1, n_2, n_3) \\
\end{align*}
and
\begin{align*}
\partial_3 a(n_1, n_2, n_3)=a(n_1, n_2, n_3+1) - a(n_1, n_2, n_3) .
\end{align*}
The following corollary is obtained by Theorem \ref{poincare}. 
The first assertion and the second assertion correspond to existences of scalar potentials and vector potentials respectively.

\begin{coro}\label{poten}
Let $a_1(n_1, n_2, n_3)$, 
$a_2(n_1, n_2, n_3)$, and $a_3(n_1, n_2, n_3)$ 
be arithmetical functions of three variables. 
\begin{itemize}
\item[(1)] There exists an arithmetical function $b(n_1, n_2, n_3)$ such that 
\begin{align*}
\begin{cases}
\partial_1 b(n_1, n_2, n_3)=a_1(n_1, n_2, n_3) \\
\partial_2 b(n_1, n_2, n_3)=a_2(n_1, n_2, n_3) \\
\partial_3 b(n_1, n_2, n_3)=a_3(n_1, n_2, n_3)\\
\end{cases}
\end{align*}
if and only if  
\begin{align*}
\partial_j a_i (n_1, n_2, n_3) = \partial_i a_j (n_1, n_2, n_3) 
\end{align*}
for any $i$ and $j$.
\item[(2)] There exist three arithmetical functions $b_1(n_1, n_2, n_3)$, 
$b_2(n_1, n_2, n_3)$, and $b_3(n_1, n_2, n_3)$ such that 
\begin{align*}
\begin{cases}
a_1(n_1, n_2, n_3)= \partial_2 b_3(n_1, n_2, n_3) - \partial_3 b_2(n_1, n_2, n_3) \\
a_2(n_1, n_2, n_3)= \partial_3 b_1(n_1, n_2, n_3) - \partial_1 b_3(n_1, n_2, n_3)   \\
a_3(n_1, n_2, n_3)= \partial_1 b_2(n_1, n_2, n_3) - \partial_2 b_1(n_1, n_2, n_3) 
\end{cases}
\end{align*} 
if and only if 
\begin{align*}
\partial_1 a_1(n_1,  n_2, n_3)+\partial_2 a_2(n_1,  n_2, n_3)+ \partial_3 a_3(n_1, n_2, n_3) =0.
\end{align*}
\end{itemize}
\end{coro}

\section{Proof of Theorem \ref{stokes}}

For 
$f  \in \mathcal{A}(S)$, we define  
$\partial_i f \in \mathcal{A}(S) (i=1, \cdots, d)$ by 
\begin{align*}
(\partial_i f )(a) 
= f(a+e_i) - f(a).
\end{align*}
Note that $\partial_i \partial_j=\partial_j \partial_i$.

First, we define exterior differential operators $D_q:\Omega^{q-1}(S) \rightarrow \Omega^{q}(S)$ $(q=1, \cdots, d)$ as follows. 
For  $f \in \Omega^0(S)$, put 
\begin{align*}
D_1(f)=\sum_{i=1}^{d} \partial_i f \; dx_i \in \Omega^1(S). 
\end{align*} 
For $\omega= \sum_{1 \leq i_1 < \cdots < i_{q-1} \leq d}f_{i_1  \cdots  i_{q-1} }
\; dx_{i_1} \wedge \cdots \wedge dx_{i_{q-1}} \in \Omega^{q-1}(S)$, put 
\begin{align*}
D_{q} \left( \omega \right)
= \sum_{1 \leq i_1 < \cdots < i_{q-1} \leq d} D_1 \left( f_{i_1  \cdots  i_{q-1} } \right) \wedge
\; dx_{i_1} \wedge \cdots \wedge dx_{i_{q-1}} \in \Omega^q(S).
\end{align*}

We have the next lemma.

\begin{lem}
 $\textrm{Im}(D_{q}) \subset \textrm{Ker}(D_{q+1}) $. 
\end{lem}

\begin{proof}
The proof is straightforward. It is the same as the theory of differential forms. 
(See \cite{bott} or \cite{madsen}.)
\end{proof}

When $d=3$,  $D_1$, $D_2$, and $D_3$ correspond to the gradient, the rotation, 
and the divergence respectively.

\begin{rem}
Let $A$ be an $R$ module and consider $R$ homomorphisms 
\begin{align*}
\partial_i: A \longrightarrow A \;\;\; (i=1, \cdots, d)
\end{align*}
such that 
$\partial_i \partial_j = \partial_j \partial_i$ for any $i$ and  $j$. 
For $q=0, \cdots, d$, put 
\begin{align*}
\Omega^{q}=\{ (f_{i_1 \cdots i_{q}})_{1 \leq  i_1 < \cdots < i_{q}\leq d} \; 
| \; f_{i_1 \cdots i_{q}} \in A \}.  
\end{align*} 
In general, we can also do the same argument about these. 
\end{rem}

Before construction boundary maps, 
we introduce some notation. 
For $I \subset \mathbb{N}$ and 
$j \in \mathbb{N}$, 
let 
\begin{align*}
s_I(j)=
\begin{cases}
+1 \;\;\; \textrm{ the order of $\{ i \in I \; |\;  i < j \}$ is even,} \\
-1 \;\;\;  \textrm{ the oder of $\{ i \in I \; |\;  i < j \}$ is odd.}
\end{cases} 
\end{align*} 
Now, we construct  $D'_q$. 
We define $D_1':N_1(S) \rightarrow N_0(S)$ by 
\begin{align*}
D_1'(E)=[a+e_l] - [a]
\end{align*}
for $E=[a:e_l]$. 
Let $q=2, \cdots, d$. 
We define $D_q':N_q(S) \rightarrow N_{q-1}(S)$ by  
\begin{align*}
&D_q'(A) \\
&= \sum_{i=1}^{q} s_I(l_i) \left([a+e_{l_i}:e_{l_1}, \cdots, e_{l_{i-1}}, e_{l_{i+1}} \cdots, e_{l_q} ] 
- [a: e_{l_1}, \cdots, e_{l_{i-1}}, e_{l_{i+1}} \cdots, e_{l_q} ]\right)
\end{align*}
for $A=[a:e_{l_1}, \cdots, e_{l_q}]$ $(l_1 < \cdots < l_q)$. Here $I=\{ l_1, \cdots, l_q \}$.
We extend these maps as $R$ linear maps.

For example, 
\begin{align*}
D_2'(S)= [a+e_1:e_2] - [a: e_2]
- [a+e_2; e_1] + [a;e_1]
\end{align*}
 for $S=[a:e_1, e_2] \in N_2(S)$. 
In addition, we see 
\begin{align*}
D_3'(C)=[a+e_1:e_2, e_3] -[a: e_2, e_3] \\
-[a+e_2:\ e_1, e_3] +[a: e_1, e_3] \\
+[a+e_3:e_1, e_2] -[a: e_1, e_2]
\end{align*}
for  $C=[a:e_1, e_2, e_3], 
\in N_3(S)$ and 
\begin{align*}
D_4'(A) 
&=[a+e_1:e_2, e_3, e_4]
- [a:e_2, e_3, e_4] \\ 
&- [a+e_2:e_1, e_3, e_4] 
+ [a:e_1, e_3, e_4] \\
&+ [a+e_3:e_1, e_2, e_4]
- [a:e_1, e_2, e_4] \\
&-[a+e_4:e_1, e_2, e_3] 
+[a: e_1, e_2, e_3] 
\end{align*}
for $A=[a: e_1, e_2, e_3, e_4], 
\in N_4(S)$.

\begin{lem}
$\textrm{Im}(D'_{q+1}) \subset \textrm{Ker}(D'_{q}) $
\end{lem}
\begin{proof}
Since boundary maps are linear, we shall show  
$D'_{q-1}D'_{q}A=0$
for $A=[a:e_{i_1}, \cdots, e_{i_q}] \in N_q(S)$. We may assume that 
$\{ i_1, \cdots, i_q\}=\{ 1, \cdots, q \}$ without loss of generality. 

First, we see 
\begin{align*}
D'_1D'_2(S) =&D_1[a+e_1:e_2] - D_1[a: e_2]
- D_1[a+e_2: e_1] + D_1[a:e_1] \\
= &\left( [a+e_1+e_2]-[a+e_1] \right)
- \left( [a+e_2]-[a] \right) \\
&-  \left( [a+e_1+e_2]-[a+e_2] \right)
+  \left( [a+e_1]-[a] \right)=0
\end{align*}
for $S=[a:e_1, e_2] \in N_2(S)$.

Next, we see 
\begin{align*}
D'D'(A) 
&= \sum_{i=1}^{q} s_I(i) \left( D' [a+e_{i}:
\overbrace{e_{1},  \cdots, e_{q}}^{\textrm{omit $i$}}] 
- D'[a: \overbrace{e_{1},  \cdots, e_{q}}^{\textrm{omit $i$}} ]\right)
\end{align*}
for $A=[a:e_1, \cdots, e_q]$.
 Here $I=\{ 1, \cdots, q \}$. 
Moreover, we see  
\begin{align*}
D'D'(A) 
&= \sum_{i=1}^{q} s_I(i)
 \sum_{j=1 \atop j \neq i}^{q} s_{I\setminus \{ i\}}(j) 
\left(
 [a+e_i+e_j:
\overbrace{e_1,  \cdots, e_q}^{\textrm{omit $i$ and  $j$}}]  
 -  [a+e_i:
\overbrace{e_1,  \cdots, e_q}^{\textrm{omit $i$ and $j$}}]\right) \\
&- 
\sum_{i=1}^{q} s_I(i)
 \sum_{j=1 \atop j \neq i}^{q} s_{I\setminus \{ i\}}(j) 
 \left( [a+e_j:
\overbrace{e_1,  \cdots, e_q}^{\textrm{omit $i$ and $j$}}]  
 -  [a:
\overbrace{e_1,  \cdots, e_q}^{\textrm{omit $i$ and $j$}}]   
\right) \\
& =: \Sigma_1-\Sigma_2-\Sigma_3+\Sigma_4. 
\end{align*} 
If we show $\Sigma_4=0$, $\Sigma_1=0$, and $\Sigma_2+\Sigma_3=0$, the assertion follows. 
Note that 
\begin{align*}
s_{I\setminus \{ i \}}(j)=
\begin{cases}
s_I(j) \;\;\; \textrm{ if $j<i$,} \\
-s_I(j) \;\;\;  \textrm{ if $j>i$}.
\end{cases} 
\end{align*}
Then we have 
\begin{align*}
\Sigma_4=&\sum_{i=1}^{q} s_I(i)
 \sum_{j=1 \atop j \neq i}^{q} s_{I\setminus \{ i\}}(j) 
 [a:
\overbrace{e_1,  \cdots, e_q}^{\textrm{omit $i$ and $j$}}]   \\
&=
\sum_{i=1}^{q} s_I(i)\left(
 \sum_{j < i} s_{I\setminus \{ i\}}(j) 
 [a:
\overbrace{e_1,  \cdots, e_q}^{\textrm{omit $i$ and $j$}}]   
+
 \sum_{j > i} s_{I\setminus \{ i\}}(j) 
 [a:
\overbrace{e_1,  \cdots, e_q}^{\textrm{omit $i$ and $j$}}] 
\right) \\
&=
\sum_{i=1}^{q} s_I(i)\left(
 \sum_{j < i} s_I(j)
 [a:
\overbrace{e_1,  \cdots, e_q}^{\textrm{omit $i$ and $j$}}]   
-
 \sum_{j > i} s_I(j) 
 [a:
\overbrace{e_1,  \cdots, e_q}^{\textrm{omit $i$ and $j$}}] 
\right) \\
&=
\sum_{i=1}^{q}
 \sum_{j < i}  s_I(i)s_I(j)
 [a:
\overbrace{e_1,  \cdots, e_q}^{\textrm{omit $i$ and $j$}}]   
-\sum_{i=1}^{q}  \sum_{j > i} s_I(i)s_I(j) 
 [a:
\overbrace{e_1,  \cdots, e_q}^{\textrm{omit $i$ and $j$}}]. 
\end{align*}
Thus, $\Sigma_4=0$. 
Similarly, we obtain 
\begin{align*}
\Sigma_1 &=\sum_{i=1}^{q} s_I(i)
 \sum_{j=1 \atop j \neq i}^q s_{I\setminus \{ i\}}(j) 
 [a+e_i+e_j:
\overbrace{e_1,  \cdots, e_q}^{\textrm{omit $i$ and $j$}}]  \\
&= 
\sum_{i=1}^{q} s_I(i)\left(
 \sum_{j < i} s_{I\setminus \{ i\}}(j) 
 [a+e_i+e_j:
\overbrace{e_1,  \cdots, e_q}^{\textrm{omit $i$ and $j$}}]   
+
 \sum_{j > i} s_{I\setminus \{ i\}}(j) 
 [a+e_i+e_j:
\overbrace{e_1,  \cdots, e_q}^{\textrm{omit $i$ and $j$}}] 
\right) \\
&= 
\sum_{i=1}^{q} s_I(i)\left(
 \sum_{j < i} s_I(j)
 [a+e_i+e_j:
\overbrace{e_1,  \cdots, e_q}^{\textrm{omit $i$ and $j$}}]   
-
 \sum_{j > i} s_I(j) 
 [a+e_i+e_j:
\overbrace{e_1,  \cdots, e_q}^{\textrm{omit $i$ and $j$}}] 
\right) \\
&=
\sum_{i=1}^{q}
 \sum_{j < i}  s_I(i)s_I(j)
 [a+e_i+e_j:
\overbrace{e_1,  \cdots, e_q}^{\textrm{omit $i$ and $j$}}]   
-\sum_{i=1}^{q}  \sum_{j > i} s_I(i)s_I(j) 
 [a+e_i+e_j:
\overbrace{e_1,  \cdots, e_q}^{\textrm{omit $i$ and $j$}}]. 
\end{align*}
Thus, $\Sigma_1=0$. 
In addition, 
\begin{align*}
&\Sigma_2+\Sigma_3 \\
=&\sum_{i=1}^{q} s_I(i)
 \sum_{j=1 \atop j \neq i}^{q} s_{I\setminus \{ i\}}(j) 
 [a+e_i:
\overbrace{e_1,  \cdots, e_q}^{\textrm{omit $i$ and $j$}}]
+ 
\sum_{i=1}^{q} s_I(i)
 \sum_{j=1 \atop j \neq i}^{q} s_{I\setminus \{ i\}}(j) 
 [a+e_j:
\overbrace{e_1,  \cdots, e_q}^{\textrm{omit $i$ and $j$}}]  \\
=& \sum_{i=1}^{q} s_I(i)
 \sum_{j<i} s_{I\setminus \{ i\}}(j) 
 [a+e_i:
\overbrace{e_1,  \cdots, e_q}^{\textrm{omit $i$ and $j$}}]
+
\sum_{i=1}^{q} s_I(i)
 \sum_{j >i} s_{I\setminus \{ i\}}(j) 
 [a+e_i:
\overbrace{e_1,  \cdots, e_q}^{\textrm{omit $i$ and $j$}}] \\
+& 
\sum_{i=1}^{q} s_I(i)
 \sum_{j <  i} s_{I\setminus \{ i\}}(j) 
 [a+e_j:
\overbrace{e_1,  \cdots, e_q}^{\textrm{omit $i$ and $j$}}] 
+
\sum_{i=1}^{q} s_I(i)
 \sum_{ j > i} s_{I\setminus \{ i\}}(j) 
 [a+e_j:
\overbrace{e_1,  \cdots, e_q}^{\textrm{omit $i$ and $j$}}].  
\end{align*}
Hence, 
\begin{align*}
&\Sigma_2+\Sigma_3 \\
=& \sum_{i=1}^{q} s_I(i)
 \sum_{j<i} s_{I}(j) 
 [a+e_i:
\overbrace{e_1,  \cdots, e_q}^{\textrm{omit $i$ and $j$}}]
-
\sum_{i=1}^{q} s_I(i)
 \sum_{j >i} s_{I}(j) 
 [a+e_i:
\overbrace{e_1,  \cdots, e_q}^{\textrm{omit $i$ and $j$}}] \\
+& 
\sum_{i=1}^{q} s_I(i)
 \sum_{j <  i} s_{I}(j) 
 [a+e_j:
\overbrace{e_1,  \cdots, e_q}^{\textrm{omit $i$ and $j$}}] 
-
\sum_{i=1}^{q} s_I(i)
 \sum_{ j > i} s_{I}(j) 
 [a+e_j:
\overbrace{e_1,  \cdots, e_q}^{\textrm{omit $i$ and $j$}}]. 
\end{align*}
Thus, $\Sigma_2+\Sigma_3=0$. 
\end{proof}

Moreover, we define integrals $B_q$.  
For $f \in \Omega^0(S)$ and $[a] \in N_0(S)$, define 
\begin{align*}
B_0 \left( f, [a] \right) = f(a).
\end{align*}
Let $q=1, \cdots, d$. 
For 
\begin{align*}
\omega=
\sum_{1 \leq i_1 < \cdots < i_q \leq d}f_{i_1  \cdots  i_q }
\; dx_{i_1} \wedge \cdots \wedge dx_{i_q} \in \Omega^q(S)
\end{align*}
and $A=[a:e_{l_1}, \cdots, e_{l_q}] \in N_q(S)$, define 
\begin{align*}
B_q(\omega, A)=f_{l_1 \cdots l_q}(a).
\end{align*}
We extend this map as bilinear map.

It remains to prove  analogues of integral theorems. 
Let $f \in \Omega^0(S)$. For 
\begin{align*}
D_1(f)=\sum_{i=1}^{d} \partial_i f dx_i  \in \Omega^1(S)
\end{align*}
and $E=[a: e_l]$, we see  
\begin{align*}
B_1(D_1f, E)= (\partial_l f)(a)=f(a+e_l)-f(a). 
\end{align*}
We also see 
\begin{align*}
B_0(f, D_1'E)=B_0(f, [a+e_l]-[a])=f(a+e_l)-f(a).
\end{align*}
Thus, $B_1(D_1f, E)=B_0(f, D_1'E)$.

For $\omega=\sum_{i_1< \cdots < i_{q-1}} f_{i_1 \cdots i_{q-1}} \; dx_{i_1} 
\wedge \cdots \wedge dx_{i_{q-1}} \in \Omega^{q-1}(S)$ and 
$A=[a:e_{l_1}, \cdots, e_{l_q}] \in N_q(S)$
$(l_1 < \cdots < l_q)$,  
we see 
\begin{align*}
D \left( \omega \right)
&= \sum_{1 \leq i_1 < \cdots < i_{q-1} \leq d} D_1 \left( f_{i_1  \cdots  i_{q-1} } \right) \wedge
\; dx_{i_1} \wedge \cdots \wedge dx_{i_{q-1}} \\
&= 
\sum_{1 \leq i_1 < \cdots < i_{q-1} \leq d}  \left( \sum_{j=1}^{d} \partial_j f_{i_1  \cdots  i_{q-1} } dx_j \right) \wedge
\; dx_{i_1} \wedge \cdots \wedge dx_{i_{q-1}} \\
&= 
\sum_{1 \leq i_1 < \cdots < i_{q-1} \leq d} 
\sum_{j \neq i_1, \cdots, i_{q-1}} 
\partial_j f_{i_1  \cdots  i_{q-1} } 
dx_j  \wedge dx_{i_1} \wedge \cdots \wedge dx_{i_{q-1}} \\
&= 
\partial_{l_1} f_{l_2  \cdots  l_{q} } 
dx_{l_1}  \wedge dx_{l_2} \wedge \cdots \wedge dx_{l_{q}} \\
&+
\partial_{l_2} f_{l_1 l_3  \cdots  l_{q} } 
dx_{l_2}  \wedge dx_{l_1} \wedge dx_{l_3} \wedge \cdots \wedge dx_{l_{q}} \\
&+ 
\partial_{l_3} f_{l_1 l_2 l_4 \cdots  l_{q} } 
dx_{l_3}  \wedge dx_{l_1} \wedge dx_{l_2} \wedge dx_{l_4} \wedge 
\cdots \wedge dx_{l_{q}} \\
&+ \cdots \\
&+\partial_{l_q} f_{l_1  \cdots  l_{q-1} } 
dx_{l_q}  \wedge dx_{l_1} \wedge dx_{l_2} \wedge dx_{l_4} \wedge 
\cdots \wedge dx_{l_{q-1}} \\
&+ 
\textrm{other terms}.
\end{align*}
Hence, 
\begin{align*}
B_q(D\omega,  A) &= \left( 
\partial_{l_1} f_{l_2  \cdots  l_{q} } 
-
\partial_{l_2} f_{l_1 l_3  \cdots  l_{q} } 
+ 
\cdots
+(-1)^{q-1}  \partial_{l_q} f_{l_1  \cdots  l_{q-1}} 
\right) (a) \\
&= 
\left( 
f_{l_2  \cdots  l_{q} }(a+e_{l_1})- f_{l_2  \cdots  l_{q} }(a) 
\right) 
-
\left( 
f_{l_1 l_3  \cdots  l_{q} } (a+e_{l_2})- f_{l_1 l_3  \cdots  l_{q} } (a) 
\right) \\
&+ 
\cdots 
+(-1)^{q-1}
\left( 
f_{l_1  \cdots  l_{q-1} } (a+e_{l_q})- f_{l_1  \cdots  l_{q-1} } (a) 
\right).
\end{align*}

We also see 
\begin{align*}
B_{q-1}(\omega, D'A) 
&=  
\left(B_{q-1}(\omega, [a+e_{l_1}:e_{l_2},  \cdots, e_{l_q} ]) 
- B_{q-1}(\omega, [a: e_{l_2},  \cdots, e_{l_q} ])\right) \\
&-
\left(B_{q-1}(\omega, [a+e_{l_2}:e_{l_1}, e_{l_{3}}, \cdots, e_{l_q} ]) 
- B_{q-1}(\omega, [a: e_{l_1}, e_{l_{3}},  \cdots, e_{l_q} ])\right) \\
&+ \cdots \\
&(-1)^{q-1}
\left(B_{q-1}(\omega, [a+e_{l_q}:e_{l_1},  \cdots, e_{l_{q-1}} ] )
- B_{q-1}(\omega, [a: e_{l_1},  \cdots, e_{l_{q-1}}]) \right) \\
 &= 
\left( 
f_{l_2  \cdots  l_{q} }(a+e_{l_1})- f_{l_2  \cdots  l_{q} }(a) 
\right) 
-
\left( 
f_{l_1 l_3  \cdots  l_{q} } (a+e_{l_2})- f_{l_1 l_3  \cdots  l_{q} } (a) 
\right) \\
&+ 
\cdots 
+(-1)^{q-1}
\left( 
f_{l_1  \cdots  l_{q-1} } (a+e_{l_q})- f_{l_1  \cdots  l_{q-1} } (a) 
\right).
\end{align*}
Thus, $B_q(D_q\omega, A)=B_{q-1}(\omega, D_q'A)$.

\section{Proof of Theorem \ref{poincare}}
Our purpose is to show an analogue of Poincar\'{e}'s lemma. 
The statements of Corollary \ref{poten} (1) and (2) correspond to $H^1(\mathbb{N}^3)=0$ and $H^2(\mathbb{N}^3)=0$ respectively. 

First, we note that  the fact $H^1(\mathbb{N}^d)=0$ can be  proved by an elementary method. 

\begin{prop}
$H^1(\mathbb{N}^d)=0$. 
\end{prop}
\begin{proof}
Let $f_i (i=1, \cdots, d)$ be  arithmetical functions of $d$ variables such that 
$\partial_i f_j = \partial_j f_i$ for any $i$ and $j$.
It is sufficient to show that there exists an arithmetical function $F(n_1, \cdots, n_d)$ of $d$ 
variables such that $\partial_i F = f_i$ for any $i$. 

Put 
\begin{align*}
F(n_1, \cdots, n_d) &=\sum_{k_1=1}^{n_1-1}f_1(k_1, 1, \cdots, 1) \\
&+ \sum_{k_2=1}^{n_2-1}f_2(n_1, k_2, 1, \cdots, 1) \\
&+ \cdots \\
&+ \sum_{k_d=1}^{n_d-1}f_d(n_1, \cdots, n_{d-1}, k_d).
\end{align*}
 
Then, we see 
\begin{align*}
(\partial_i F)(n_1, \cdots, n_d) &=
\partial_i \left( 
\sum_{k_i=1}^{n_i-1}f_i(n_1, \cdots, n_{i-1}, k_i, 1, \cdots, 1)  \right) \\
&+ 
\sum_{k_{i+1}=1}^{n_{i+1}-1}(\partial_i f_{i+1})(n_1, \cdots, n_i, k_{i+1}, 1, \cdots, 1) 
 \\
&+ \cdots \\
&+ 
\sum_{k_d=1}^{n_d-1}(\partial_i f_d)(n_1, \cdots, n_{d-1}, k_d). 
\end{align*}
By the assumption, we see 
\begin{align*}
(\partial_i F)(n_1, \cdots, n_d) = 
&f_i(n_1, \cdots, n_i, 1, \cdots, 1)   \\
&+ 
\sum_{k_{i+1}=1}^{n_{i+1}-1}(\partial_{i+1} f_{i})(n_1, \cdots, n_i, k_{i+1}, 1, \cdots, 1) 
 \\
&+ \cdots \\
&+ 
\sum_{k_d=1}^{n_d-1}(\partial_d f_i)(n_1, \cdots, n_{d-1}, k_d) \\
= 
&f_i(n_1, \cdots, n_i, 1, \cdots, 1)   \\
&+ 
 f_{i}(n_1, \cdots, n_{i+1}, 1, \cdots, 1) -
f_{i}(n_1, \cdots, n_i, 1, \cdots, 1)
 \\
&+ \cdots \\
&+ 
 f_i(n_1, \cdots, n_d) - 
 f_i(n_1, \cdots, n_{d-1}, 1). 
\end{align*}
Thus $\partial_i F = f_i$. 
\end{proof}

It remains to show that $H^q(\mathbb{N}^d)=0$ for $q \geq 2$. 
Let us show that 
$H^q(\mathbb{N}^{d+1})$ and $H^q(\mathbb{N}^d)$ are isomorphic as $R$ modules. 
The proof is based on the similar method in  \cite{bott} or \cite{madsen}. 

For simplicity of notation, if 
$I=\{i_1, \cdots, i_q\}$ $(i_1 < \cdots< i_q)$, we write $f(n_1, \cdots, n_d)dx_I$
 instead of $f(n_1, \cdots, n_d)dx_{i_1}\cdots dx_{i_q}$.

We define $\pi:\mathbb{N}^{d+1}\rightarrow\mathbb{N}^d$ by 
$\pi(n_1, \cdots, n_d, t)=(n_1, \cdots, n_d)$ and 
 $\pi^*:\Omega^{q}(\mathbb{N}^{d})\rightarrow \Omega^{q}(\mathbb{N}^{d+1})$ by 
\begin{align*}
\pi^{*}(\omega)=  f \circ \pi (n_1, \cdots, n_d, t) 
dx_I 
\end{align*}
where 
\begin{align*}
\omega=f(n_1, \cdots, n_d) 
dx_I. 
\end{align*}

We define $s:\mathbb{N}^{d}\rightarrow\mathbb{N}^{d+1}$ by
$s(n_1, \cdots, n_d)=(n_1, \cdots, n_d, 1)$. 
In addition, we define  
$s^*:\Omega^{q}(\mathbb{N}^{d+1})\rightarrow \Omega^{q}(\mathbb{N}^{d})$ by 
\begin{align*}
s^{*}(\omega)=  
\begin{cases}
f \circ s (n_1, \cdots, n_d) 
dx_I \;\;\; \textrm{if $d+1 \not\in  I$,} \\
0 \;\;\; \textrm{if $d+1 \in  I$}
\end{cases}
\end{align*}
where
\begin{align*}
\omega=f(n_1, \cdots, n_d, t) 
dx_I.  
\end{align*}

An easy verification shows the following lemma. 

\begin{lem}
$\pi^*$ and $s^*$ are compatible with $D$.  
(In particular, a closed form is mapped to a closed form by $\pi^*$ or $s^*$.)
\end{lem}

Next, we define $K:\Omega^{q}(\mathbb{N}^{d+1}) \rightarrow \Omega^{q-1}(\mathbb{N}^{d+1})$ as follows. 
Let  
\begin{align*}
\omega= f(n_1, \cdots, n_d, t)dx_{i_1} \wedge dx_{I'}\in \Omega^{q}(\mathbb{N}^{d+1}).
\end{align*}
Here $I'=\{ i_2, \cdots, i_{q} \}$. 
If $d+1 \not\in \{ i_1 \} \cup I'$, we put
\begin{align*}
K(\omega)=0.
\end{align*} 
If $i_1=d+1$ and  $d+1 \not\in I'$, we put 
\begin{align*}
K(\omega)= \left( \sum_{k=1}^{t-1}f(n_1, \cdots, n_d, k) \right)
dx_{I'}.
\end{align*}

It is clare that $s^* \circ \pi^*$ is the identity map of $\Omega^q(\mathbb{N}^{d})$.
Conversely, $\pi^*\circ s^*$ is not the identity map of $\Omega^q(\mathbb{N}^{d+1})$. 
However, 
$\textrm{Id}-\pi^* \circ s^*$ is zero on $H^q(\mathbb{N}^{d+1})$ by the next proposition. 
Therefore, we regard 
$\pi^* \circ s^*$ and  $s^* \circ \pi^*$ as the identity maps of $H^q(\mathbb{N}^{d+1})$ and $H^q(\mathbb{N}^{d})$ respectively. 
Thus, $H^q(\mathbb{N}^{d+1})$ and $H^q(\mathbb{N}^{d})$ are isomorphic. 
We can see Theorem \ref{poincare} inductively.

\begin{prop}\label{pisdk} 
We have 
$Id-\pi^* \circ s^* = DK+KD$ on $\Omega^q(\mathbb{N}^{d+1})$
\end{prop}
\begin{proof}

For 
\begin{align*}
\omega=f(n_1, \cdots, n_d, t)dx_I \in \Omega^{q}(\mathbb{N}^{d+1}), 
\end{align*}
we see 
\begin{align*}
(Id-\pi^* \circ s^*)(\omega)=
\begin{cases}
\{ f (n_1, \cdots, n_d, t) - f (n_1, \cdots, n_d, 1)\} dx_I \;\;\;\; 
\textrm{if $d+1 \not\in I$,} \\ 
f(n_1, \cdots, n_d, t)dx_I 
\;\;\;\; 
\textrm{if $d+1 \in I$}. 
\end{cases}
\end{align*}

First, let 
\begin{align*}
\omega= f(n_1, \cdots, n_d, t)dx_I \in \Omega^q(\mathbb{N}^{d+1})
\end{align*}
with $d+1 \not\in I$. 
It is evident that $DK(\omega)=0$ by definition of $K$.
Since 
\begin{align*}
KD(\omega) &= K \left( 
\partial_{d+1} f(n_1, \cdots, n_d, t) dx_{d+1} \wedge dx_I 
+ \textrm{other tems} \right) \\ 
&= \sum_{k=1}^{t-1}
\{ f(n_1, \cdots, n_d, k+1) - f(n_1, \cdots, n_d, k) \} dx_I 
 \\
&=
\{  f(n_1, \cdots, n_d, t) - f(n_1, \cdots, n_d, 1) \} dx_I, 
\end{align*}
we have 
\begin{align*}
(DK+KD)(\omega)=\{  f(n_1, \cdots, n_d, t) - f(n_1, \cdots, n_d, 1) \} dx_I.
\end{align*}

Second, let 
\begin{align*}
\omega= g(n_1, \cdots, n_d, t)dx_{i_1} \wedge dx_{I'}
 \in \Omega^{q}(\mathbb{N}^{d+1}).
\end{align*}
Here $I'=\{ i_2, \cdots, i_{q} \}$ and suppose that $i_1=d+1$ and  $d+1 \not\in I'$. 
We have 
\begin{align*}
KD(\omega) &= K \left(\sum_{j=1}^{d+1} 
\partial_{j} g(n_1, \cdots, n_d, t) dx_j \wedge dx_{d+1}\wedge dx_{I'} \right) \\
&= -K \left(\sum_{j=1 \atop j \not\in \{ d+1 \} \cup I'}^{d+1} 
\partial_{j} g(n_1, \cdots, n_d, t) dx_{d+1} \wedge dx_{j}\wedge dx_{I'} \right) \\
&= - \sum_{j=1 \atop j \not\in \{ d+1 \} \cup I'}^{d+1}  \sum_{k=1}^{t-1}
\partial_{j} g(n_1, \cdots, n_d, k)   dx_{j}\wedge dx_{I'}
\end{align*}
and
\begin{align*}
DK(\omega) &= D\left( \sum_{k=1}^{t-1}g(n_1, \cdots, n_d, k) dx_{I'}\right) \\
&= \sum_{j=1}^{d+1}\partial_j \left( \sum_{k=1}^{t-1}g(n_1, \cdots, n_d, k) \right)
dx_j \wedge dx_{I'}
\\
&= \sum_{j=1 \atop j \not\in \{ d+1 \} \cup I'}^{d+1}  \sum_{k=1}^{t-1}
\partial_{j} g(n_1, \cdots, n_d, k)   dx_{j}\wedge dx_{I'}
+ g(n_1, \cdots, n_d, t)   dx_{d+1}\wedge dx_{I'}.
\end{align*}
Thus, 
\begin{align*}
(DK+KD)(\omega)=g(n_1, \cdots, n_d, t)   dx_{d+1}\wedge dx_{I'}.
\end{align*}
\end{proof}

\end{document}